\documentclass[reqno,12pt]{amsart}

\usepackage{color} 
\usepackage{ifpdf}
\ifpdf
    \usepackage[pdftex]{graphicx}
    \usepackage[pdftex]{hyperref}
    \hypersetup{
        unicode=false,          
        pdftoolbar=true,        
        pdfmenubar=true,        
        pdffitwindow=false,     
        pdfstartview={FitH},    
        pdftitle={SP Article},      
        pdfauthor={Sara Pollock},   
        pdfsubject={Mathematics},    
        pdfcreator={Sara Pollock},  
        pdfproducer={Sara Pollock}, 
        pdfkeywords={PDE, analysis, numerical analysis}, 
        pdfnewwindow=true,      
        colorlinks=true,        
        linkcolor=red,          
        citecolor=blue,         
        filecolor=magenta,      
        urlcolor=cyan           
    }

\else
    \usepackage{graphicx}
    \usepackage{pstricks}
    
    \newcommand{\href}[2]{#2}
\fi

\usepackage{times}
\usepackage{amsfonts}
\usepackage[font={small}]{caption}

\usepackage{amsmath}
\usepackage{amsthm}
\usepackage{amssymb}
\usepackage{amsbsy}
\usepackage{amscd}
\usepackage{extarrows}

\usepackage{enumerate}
\usepackage{verbatim}
\usepackage{subfigure}

\newtheorem{theorem}{Theorem}[section]
\newtheorem{corollary}[theorem]{Corollary}
\newtheorem{lemma}[theorem]{Lemma}

\newtheorem{assumption}[theorem]{Assumption}

\newtheorem{definition}[theorem]{Definition}

\newtheorem{remark}[theorem]{Remark}

\numberwithin{equation}{section}  

\definecolor{myblue}{rgb}{0.2,0.2,0.7}
\definecolor{mygreen}{rgb}{0,0.6,0}
\definecolor{mycyan}{rgb}{0,0.6,0.6}
\definecolor{myred}{rgb}{0.9,0.2,0.2}
\definecolor{mymagenta}{rgb}{0.9,0.2,0.9}
\definecolor{mywhite}{rgb}{1.0,1.0,1.0}
\definecolor{myblack}{rgb}{0.0,0.0,0.0}

\newcommand{\tdot}{\!\cdot\!}
\renewcommand{\div}{{\operatorname{div}}}
\newcommand{\dif}{{\operatorname{d}}}

\newcommand{\R}{{\mathbb R}}       
\newcommand{\cD}{{\mathcal D}}
\newcommand{\cI}{{\mathcal I}}
\newcommand{\cT}{{\mathcal T}}
\newcommand{\cV}{{\mathcal V}}
\DeclareMathAlphabet{\mathpzc}{OT1}{pzc}{m}{it}

\newcommand{\on}{\text{ on }}
\newcommand{\an}{\text{ and }}
\newcommand{\with}{\text{\ with }}
\newcommand{\inn}{\text{ in }}
\newcommand{\tforall}{\text{ for all }}
\newcommand{\ale}{\text{a.e.\,}}
\definecolor{blue}{rgb}{0.2,0.2,0.7}
\definecolor{red}{rgb}{0.7,0.3,0.1}
\definecolor{cyan}{rgb}{0.2,0.5,0.6}
\usepackage{mathtools}
\usepackage{stmaryrd}

\usepackage{esint}
\begin{document}

\title[Discrete comparison principles]
      {Discrete comparison principles for quasilinear elliptic PDE}

\author[S. Pollock]{Sara Pollock}
\email{sara.pollock@wright.edu}

\author[Y. Zhu]{Yunrong Zhu}
\email{zhuyunr@isu.edu}

\address{Department of Mathematics and Statistics\\
         Wright State University\\ 
         Dayton, OH 45435}

\address{Department of Mathematics and Statistics\\
         Idaho State University\\ 
         Pocatello, ID 83209}
\thanks{SP was supported in part by NSF DMS 1719849. YZ was supported in part by NSF DMS 1319110.}

\date{\today}

\keywords{
Discrete comparison principle,
uniqueness,
nonlinear diffusion,
quasilinear partial differential equations
}


\subjclass[2010]{65N30, 35J62, 35J93}

\begin{abstract}
Comparison principles are developed for discrete quasilinear elliptic partial 
differential equations.
We consider the analysis of a class of nonmonotone Leray-Lions problems featuring 
both nonlinear solution and gradient dependence in the principal coefficient, 
and a solution dependent lower-order term.  
Sufficient local and global conditions on the discretization are found for 
piecewise linear finite element solutions to satisfy a comparison principle, 
which implies uniqueness of the solution. For problems without a lower-order term,
our analysis shows the meshsize is only required to be locally controlled, 
based on the variance of the computed solution over each element.
We include a discussion of the simpler semilinear case where a linear algebra 
argument allows a sharper mesh condition for the lower order term.
\end{abstract}

\maketitle

\section{Introduction}\label{sec:intro} 
We consider the finite element approximation of the quasilinear elliptic partial 
differential equation (PDE)
\begin{align}\label{eqn:PDEstrong}
-\div (a(x,u,\nabla u)) + b(x,u) = 0 ~\inn~ 
\Omega \subset \R^d,
\end{align}
where $a(x,\eta,\xi) = A(x,\eta,\xi)\xi$,
for scalar-valued $A: \Omega \times \R \times \R^d \rightarrow \R$, 
              and $b: \Omega \times \R \rightarrow \R$.
The domain $\Omega$ is assumed to be piecewise polygonal for $d = 2$, or an interval
for $d = 1$. The boundary $\partial \Omega$ 
is decomposed into Dirichlet and Neumann parts, where
the Dirichlet part $\Gamma_D \subseteq \partial \Omega$ 
has positive measure in $\R^{d-1}$, and the Neumann part
is given by $\Gamma_N = \partial \Omega \setminus \Gamma_D$.  
The boundary conditions applied to \eqref{eqn:PDEstrong} are either
mixed Dirichlet/Neumann or homogeneous Dirichlet, given by
\begin{align}\label{eqn:BC}
u = 0 \text{ on } \Gamma_D, ~\an~ a(x,u,\nabla u)\cdot n(x) = \psi(x)  
      \text{ on } \Gamma_N,
\end{align} 
for outward facing normal $n$.
The aim of this paper is to extend the discrete comparison principle and uniqueness 
results recently obtained by the authors to a more general class of quasilinear
elliptic equations. 

Significant progress has been made on
developing discrete maximum principles for  
divergence form quasilinear elliptic problems, as in 
\cite{DiKrSc13,JuUn05a,KaKo05,KKK07,WaZh12},
and developing the appropriate conditions on the angles of the mesh for these
results to hold.  
In the nonlinear context, comparison principles rather than maximum principles
for a given equation imply the uniqueness of solutions.
Comparison principles also provide important information such as a natural
ordering of solutions that can be useful in the analysis of numerical solutions.
There are still only few results on discrete comparison principles for
problem \eqref{eqn:PDEstrong}, despite the significant literature on corresponding 
results for continuous problems, {\em e.g.,} 
\cite{BBGK99,BaMu95,DoDuSe71,HlKrMa94,Trudinger74}, 
and  \cite[Chapter 10]{GilbargTrudinger}, and the references therein.  

To our knowledge, the first comparison theorem and global uniqueness result 
for a discrete version of problems in this class that holds as the mesh is refined, 
is that of \cite{AnCh96a}, 
for the equation $-\div(\kappa(x,u)\nabla u) = f(x)$, where
both a uniformly small meshsize and an acuteness condition on the angles of the 
mesh were used.  
Uniqueness of solutions for simplicial and rectangular elements of arbitrary order
with numerical quadrature
was later established in \cite{AbVi12} for this class of problems, 
dependent on a uniformly small meshsize.
The meshsize assumption for $P_1$ elements was relaxed in a comparison theorem
framework in recent work by the current authors in \cite{PoZh17a}, 
where the global meshsize condition was replaced by a local condition
on the maximum variance of the solution over each element, locally limiting
the meshsize where the solution has steep gradients.

The main contributions of the current manuscript are that we now allow a more 
general diffusion coefficient, including a nonlinear dependence on the gradient; 
and, a (nonlinear) solution-dependent lower order term. 
These results allow the determination of whether the solution to a finite element
approximation of \eqref{eqn:PDEstrong} is unique, based only on knowledge of 
problem data, and accessible properties of the computed solution and the mesh.
This information is useful in the analysis of adaptive algorithms 
({\em e.g.,} \cite{Pollock15a,Pollock15b}), and can 
be used to verify the uniqueness of a discrete solution upon numerical convergence.
Importantly, these results hold globally, as opposed to locally, within the neighborhood
of a given solution; and, without {\em a priori} knowledge of the
solution to \eqref{eqn:PDEstrong}.  
\subsection{Problem class}
\label{subsec:class}
The following assumptions on the diffusion coefficient, by means of the function 
$a_i(x,\eta,\xi) = A(x,\eta, \xi)\xi_i$, $i = 1, \ldots, d$, 
for $x \in \Omega$, $\eta \in \R$,
and $\xi \in \R^d$,
are made throughout the remainder of the paper.
\begin{assumption}\label{A1:GT}
Assume $a(x,\eta,\xi)$ and $b(x,\eta)$ are Carath\'eodory functions, 
$C^1$ in $(\eta,\xi)$ (respectively, $\eta$) for almost every (a.e.) 
$x \in \Omega$, and measurable in $x$ for each $(\eta, \xi) \in \R \times \R^d$,
(respectively, for each $\eta \in \R$).
Assume $a$ is elliptic in the following sense.
There is a positive constant $\gamma_a$ with
\begin{align}\label{eqn:A1:elliptic}
\sum_{i = 1}^d \frac{\partial a_i}{\partial \xi_j}(x,\eta,\xi) \zeta_i \zeta_j \ge \gamma_a |\zeta|^2.
\end{align} 
for $\ale x \in \Omega$, and for all $\eta \in \R$, $\xi \in \R^d$ and $\zeta \in \R^d$.
There is a constant $K_\eta > 0$ with
\begin{align}\label{eqn:A1:a_eta}
\left| \frac{\partial A}{\partial \eta}(x,\eta,\xi) \right| \le K_\eta, 
\end{align}
for $\ale x \in \Omega$ and for all $\eta \in \R$ and $\xi \in \R^d$.
Assume $b$ is nondecreasing in $\eta$, and there is a constant $B_\eta \ge 0$ with
\begin{align}\label{eqn:A1:b_eta}
0 \le \frac{\partial b}{ \partial \eta} (x,\eta) \le B_\eta,
\end{align}
for $\ale x \in \Omega$ and $\eta \in \R$.
\end{assumption}
The conditions of Assumption \ref{A1:GT}, used here to show a comparison
theorem and uniqueness of the discrete solution, also satisfy the hypotheses 
of Theorem 10.7 of \cite{GilbargTrudinger}, under condition (ii), which shows 
a comparison theorem for the continuous problem. 

\begin{remark}[Existence of solutions]
\label{rem:unique}
To understand existence of the PDE solution, it is useful to consider the Leray-Lions
and coercivity conditions (see for example \cite[Chapter 2]{CarlLeMontreanu}).
In addition to the Carath\'eodory assumption above, 
the following conditions assure the pseudo-monotonicity of the principal
part of the elliptic operator.  
\begin{enumerate}
\item[(1)] Growth condition: there is a function $k_0(x) \in L{q}(\Omega)$ and 
$c_0 > 0$ with
\[
|A(x,\eta,\xi)\xi_i| \le k_0(x) + c_0 (|\eta|^{p-1} + |\xi|^{p-1}), ~i = 1, \ldots, d,
\]
with $1<p<\infty$ and $1/p + 1/q = 1$.
\item[(2)]  Monotonicity with respect to $\xi$: the coefficients $a_i = A \xi_i$ 
satisfy
\begin{align*}
\sum_{i = 1}^d (A(x,\eta,\xi)\xi_i - A(x,\eta,\bar \xi)\bar \xi_i)(\xi_i - \bar \xi_i)
  > 0,
\end{align*}
for $\ale x \in \Omega,$ all $\eta \in \R$, and for all $\xi, \bar \xi \in \R^d$ with $\xi\ne \bar \xi.$
\item[(3)] 
Coercivity: there is a constant $\nu > 0$ and a function $k(x) \in L^{1}(\Omega)$ with
\[
\sum_{i = 1}^d A(x,\eta,\xi)\xi_i^2 \ge \nu |\xi|^p - k(x),
\]
for $\ale x \in \Omega$, all $\eta \in \R$ and all $\xi \in \R^d$.
\end{enumerate}

Classes of problems satisfying the above conditions are well-studied
in the literature with respect to existence of solutions and their boundedness
properties.  For instance, existence of solutions is shown in 
Chapter II.6 of \cite{Showalter97}, under the strengthened coercivity condition 
and additional growth condition on the lower order term
\begin{align*}
	\sum_{i=1}^{d} A(x,\eta,\xi)\xi_{i}^2 & \ge c_r |\xi|^p - K_r (K(x) + |\eta|^r) \\
b(x,\eta)            & \le K_r(k_0(x) + |\eta|^r), 
\end{align*} 
for $k_0(x)$ from condition (1) above, 
some $1 \le r < p$ and $K(x) \in L^{1}(\Omega)$ (see \cite[Lemma 6.4]{Showalter97}).

Cases where both Assumptions \ref{A1:GT} and 
conditions (1)-(3) above are satisfied are not uncommon.
First, if in addition to Assumption \ref{A1:GT}, there are constants
$0 < \lambda_A \le \Lambda_A$ with  $\lambda_A \le A(x,\eta,\xi) \le \Lambda_A$,  
then conditions (1)-(3) hold with $p = q = 2$.
This includes the case where $a_i(x,\eta,\xi) = A(x,\eta)\xi_i$, 
as in the earlier investigation \cite{PoZh17a}, with  $b\equiv 0$, which features
applications to nonlinear heat conduction, for example \cite{HlKrMa94}.
More generally, these conditions hold if $A(x,\eta,\xi)$ has the
form $A(x,\eta,\xi) = A_0(x,\eta) + A_1(x,\eta)f(|\xi|) + A_2(x)g(|\xi|)$, 
where $A_0$ is bounded away from zero, and $f(|\xi|)$ and  $g(|\xi|)$ satisfy 
appropriate growth conditions.  Problems
of this form will be specifically considered in the discrete two dimensional case. 
\end{remark}

The discrete equations for monotone instances of the above classes, 
those in which the principal coefficent is independent of $\eta$, 
such as the $p$-Laplacian, are analyzed in for instance 
\cite{BaLi93,BDK12,DiKrSc13}, and under stronger monotonicity and Lipschitz assumptions 
in \cite{CoWi17,GaMoZu11}, 
exploiting the variational structure of the problem to establish uniqueness without
a comparison principle. 
A more general approximation strategy using a Hybrid High-Order method is presented in 
\cite{DPiDr17}.  In that setting, strong convergence of the sequence of discrete 
solutions is found as the meshsize 
goes to zero for monotone problems, but the result holds only up to a subsequence
if $a(x,\eta,\xi)$ maintains its $\eta$-dependence, {\em i.e.}, 
for nonmonotone problems (see \cite[Theorem 4.6]{DPiDr17}).
The emphasis of this article is establishing 
verifiable sufficient conditions for the uniqueness of the discrete solution 
for the case where $a(x,\eta,\xi)$ of \eqref{eqn:PDEstrong} 
maintains its $\eta$-dependence, and is not then
monotone (or variational, see \cite{HlKrMa94}),
but rather {\em pseudo-monotone}, as described above.  

The weak form of \eqref{eqn:PDEstrong} is given by integration against test functions 
$v$ which lie in an appropriate subspace 
$V_{0,D} \subset V \subseteq H^{1}\cap W^{{1,p}}$,
where
$V_{0,D} = \{v \in V ~\big|~ v = 0 \text{ on } \Gamma_D\}$, and $p$ is 
determined by the particular problem class, as in
Remark \ref{rem:unique}. The reader is referred to 
\cite[\S 3.2]{CarlLeMontreanu} for detailed discussion on the existence and comparison results for the continuous Dirichlet problem.
Then, the weak form of the problem is: find $u \in V_{0,D}$ such that
\begin{align}\label{eqn:PDEweak}
\int_{\Omega} a(x,u,\nabla u) \cdot \nabla v + b(x,u)v \, \dif x
= \int_{\Gamma_N}\psi(x) v \, \dif s,
~\tforall v \in V_{0,D}, 
\end{align}
where the Neumann data $\psi(x)$ is assumed to be bounded and 
measurable.
For the remainder of the paper, we proceed with conditions of Assumptions \ref{A1:GT},
and investigate the conditions under which a discrete comparison principle holds, assuming
the existence of a discrete subsolution and supersolution, as defined in the 
next section. 

The remainder of the article is structured as follows.  In \S \ref{sec:comparison}, 
we state the discretization, and introduce the framework for proving the 
discrete comparison principle. In \S \ref{sec:1D}, this framework is applied to
the simple case of the one dimensional problem.  Then, in \S \ref{sec:2D}, the
two dimensional problem is considered.  First, additional restrictions on the
discretization (angle conditions) are introduced. 
Then, in \S \ref{subsec:2Dsetup}, useful estimates for the technical lemmas of 
\S \ref{subsec:2Dlemmas} are reviewed.  The main 2D result, Theorem
\ref{thm:2Dcompare}, follows in \S \ref{subsec:2Dcompare}.
In \S \ref{sec:semilinear} we prove a comparison principle for a simpler semilinear problem
based on the previous estimates.
In Theorem \ref{thm:semi}, we then apply a linear algebraic approach to improve the mesh condition.
\section{Overview of comparison framework}\label{sec:comparison}
We next overview the discretization and the comparison theorem framework. 
The subsequent sections contain the precise results and technical proofs.
The cases of one and two dimensions are worked out separately to give explicit 
constants that can be used as criteria for verifying uniqueness of a discrete 
solution on a given mesh.

\subsection{Discretization}
\label{subsec:discretization}
Let $\cT$ be a conforming partition of domain $\Omega$ that exactly captures the 
boundary of $\Omega$, and each of $\Gamma_D$ and $\Gamma_N$.  In one dimension,
$\cT$ is a collection of intervals, and in two dimensions a triangulation. 
Let $\overline \cD$ be the collection of vertices or nodes of $\cT$, and let
$\cD =\overline \cD \setminus \Gamma_D$.
The nodes $a \in \cD$ correspond to the mesh degrees of freedom.  
Let $\cV \coloneqq \cV_{0,D} \subset V_{0,D}$ be the discrete space spanned by
the piecewise linear basis functions $\{\varphi_j\}$ that satisfy 
$\varphi_i(a_j) = \delta_{ij}$ for each $a_j \in \cD$.  

For simplicity of defining the finite element solution space, the discussion assumes a 
a homogeneous Dirichlet part under either the mixed or pure Dirichlet conditions.
The method of the proof trivially generalizes to allow nonhomogeneous 
bounded measurable Dirichlet data, as its contribution is subtracted off
as is the Neumann data, on the first step.

\subsection{Discrete comparison framework}
\label{subsec:framework}
The discrete Galerkin problem for $\cV$ is: find $u \in \cV$
such that 
\begin{align}\label{eqn:PDEdiscrete}
\int_{\Omega} a(x,u,\nabla u) \cdot \nabla v + b(x,u)v \, \dif x
  = \int_{\Gamma_N}\psi(x) v \, \dif s,
~\tforall v \in \cV. 
\end{align}
A \emph{subsolution} to \eqref{eqn:PDEdiscrete} is a function $u_1 \in \cV$ with 
\begin{align}\label{eqn:subsoln}
\int_{\Omega} a(x,u_1,\nabla u_1) \cdot \nabla v + b(x, u_1)v \, \dif x- 
  \int_{\Gamma_N}\psi(x) v \, \dif s\le 0,
\end{align}
for all $v \in \cV^+ = \{v \in \cV ~\big|~ v \ge 0\}$. 
A corresponding \emph{supersolution} $u_2 \in \cV$ is given by
\begin{align}\label{eqn:supsoln}
\int_{\Omega} a(x,u_2,\nabla u_2) \cdot \nabla v + b(x,u_2)v \, \dif x 
  - \int_{\Gamma_N}\psi(x) v \, \dif s \ge 0,
~\tforall v \in \cV^+.
\end{align}
Subtracting \eqref{eqn:supsoln} from \eqref{eqn:subsoln}, we find
\begin{align}\label{eqn:subsup}
\int_{\Omega} (a(x,u_1,\nabla u_1) - a(x,u_2,\nabla u_2)) \cdot \nabla v 
  + (b(x, u_1)v - b(x,u_2) )v \, \dif x\le 0,
\end{align}
for all $v \in \cV^+$. 
Decomposing the principal part by $a(x,u,\nabla u) = A(x,u,\nabla u)\nabla u$, 
and applying Taylor's theorem, 
it holds for $w = u_1 - u_2$ that
\begin{align}\label{eqn:decomp_a}
&(a(x, u_1, \nabla  u_1) - a(x, u_1, \nabla u_2)) 
 + (A(x, u_1, \nabla  u_2)  - A(x, u_2, \nabla u_2)) \nabla u_2 
\nonumber \\ 
& = \int_{0}^1\frac{\partial a}{\partial \xi}(x, u_1, \nabla z(t)) 
  \nabla w\, \dif t
 + \int_{0}^1\frac{\partial A}{\partial \eta }(x, z(t), \nabla u_2) w \nabla u_2 
  \, \dif t,
\end{align}
for $z(t) = t u_1 + (1-t)u_2$. Similarly for the lower order term
\begin{align}\label{eqn:decomp_b}
b(x,u_1) - b(x,u_2) = \int_0^1\frac{\partial b}{\partial \eta}(x,z(t)) w \, \dif t.
\end{align}
Applying \eqref{eqn:decomp_a} and \eqref{eqn:decomp_b} to \eqref{eqn:subsup},
and breaking the integral over the global domain into a sum of integrals 
over each element $T \in \cT$, obtain
\begin{align}\label{eqn:decomp_T}
&\int_\Omega \int_0^1 
  \left( \frac{\partial a}{\partial \xi}(x, u_1, \nabla z(t)) \nabla w\right)\tdot \nabla v
 + \left(\frac{\partial A}{\partial \eta }(x, z(t), \nabla u_2) w\right) \nabla u_2\tdot \nabla v
\nonumber \\
& + \frac{\partial b}{\partial \eta}(x,z(t)) w v \,\dif t \, \dif x
\nonumber \\
& = \sum_{T \in \cT}\int_T \int_0^1 
  \left( \frac{\partial a}{\partial \xi}(x, u_1, \nabla z(t)) \nabla w\right) \tdot \nabla v
 + \left( \frac{\partial A}{\partial \eta }(x, z(t), \nabla u_2)\, w \right) \nabla u_2  \tdot \nabla v
\nonumber \\
& + \frac{\partial b}{\partial \eta}(x,z(t)) w v \, \dif t \, \dif x \le 0,
\end{align}
for all $v \in \cV^+$.
The structure of $a(x,u,\nabla u) = A(x,u,\nabla u)\nabla u$ is exploited in the 
first  term of the above decomposition to yield a quantity that is strictly positive,
and in the second term to create a quantity controlled by the difference in nodal
values of $u_2$. This factorization is a key component of the problem class that 
allows a condition
for uniqueness similar to that in \cite{PoZh17a}, dependent on the variance of
the discrete solution $u$ over each element. 

The proof of the comparison principle follows by considering a particular test
function $v \in \cV^+$, and finding under Assumption \ref{A1:GT} and 
additional assumptions on the discretization, that if $w > 0$ anywhere, 
the left hand side integration over elements of \eqref{eqn:decomp_T} 
is strictly positive, yielding a contradiction and implying $w \le 0$ everywhere, 
hence $u_1 \le u_2$ in $\Omega$.
Common test functions for this purpose in the continuous context include the 
positive part of $w = u_1 - u_2$,
possibly taken to some power, as in \cite{BBGK99,BaMu95}
. In the discrete setting, the positive part of $w$ is generally not a member
of the finite element space, so a discrete version of this function can be used, 
as in \cite{WaZh12}. In this case, as 
in \cite{AnCh96a,PoZh17a}, it is convenient to define a simpler test function 
$v$ as follows.
\begin{definition}\label{def:test}
Let $u_1 \in \cV$ be a subsolution of \eqref{eqn:PDEdiscrete} as in 
\eqref{eqn:subsoln}, and let $u_2 \in \cV$ be a supersolution as in 
\eqref{eqn:supsoln}. Let $w = u_1 - u_2 \in \cV$.   
Define the test function $v \in \cV^+ \subset \cV$ by its nodal values at each 
$a \in \cD$ as
\begin{align}\label{eqn:1Dtest}
v(a) = \left\{ \begin{array}{ll}1, & w(a) > 0, \\
                                0, & w(a) \le 0. \end{array} \right.
\end{align}
\end{definition}
If $w > 0$ anywhere on $\Omega$, then $v(a)$ is nonzero for some $a \in \cD$.
One of the convenient properties of this test function $v$, is that $\nabla v = 0$
over each $T \in \cT$ where $w$ does not change sign. In fact, for the 1D case,
an even simpler test function can be defined for which $v'$ is supported over
no more than two elements.  This strategy was used in \cite{PoZh17a}; however,
in this presentation we will use the same Definition \eqref{def:test} for 
both one and two dimensions to unify the arguments.

Partition the sets $\cT_+, \cT_-$ and $\cT_c$ by the value of $v$ from 
Definition \ref{def:test}, restricted to each element in $\cT$.
\begin{align}\label{eqn:Tcdef}
\cT_+ = \{ T \in \cT ~\big|~ v(x)\big|_T \equiv 1\}, \quad
\cT_- = \{ T \in \cT ~\big|~ v(x)\big|_T \equiv 0\}, \quad
\cT_c = \cT \setminus \{\cT_+ \cup \cT_-\}.
\end{align}
Write the integral over $\Omega$ as 
\[
\int_\Omega = \int_{\bigcup_{T \in \cT_+}} + \int_{\bigcup_{T \in \cT_-}}
+ \int_{\bigcup_{T \in \cT_c}}.
\]
Each integral over $T \in \cT_-$ is trivially zero.  Each integral over 
$\cT \in \cT_+$ satisfies $\nabla v \equiv 0$, and the remaining lower order
part is nonnegative by
\begin{align}\label{eqn:TinTplus}
\int_{T\in \cT_+} \int_0^1\frac{\partial b}{\partial \eta} (x,z(t)) wv \, \dif t \, \dif x = 
\int_{T\in \cT_+} \int_0^1\frac{\partial b}{\partial \eta} (x,z(t)) w \, \dif t \, \dif x \ge 0, 
\end{align}
as $w > 0$, $v = 1$ and $\partial b/\partial \eta \ge 0$, by \eqref{eqn:A1:b_eta} 
of Assumption \ref{A1:GT}.
It remains then to bound the integrals over $T \in \cT_c$ where $w$ changes sign.
In summary, we have from \eqref{eqn:subsup}, \eqref{eqn:decomp_T} and 
\eqref{eqn:TinTplus} that
\begin{align}\label{eqn:red_subsup}
0 & \ge \int_{\Omega} (a(x,u_1,\nabla u_1) - a(x,u_2,\nabla u_2)) \cdot \nabla v 
  + (b(x, u_1)v - b(x,u_2) )v \, \dif x 
\nonumber \\
  & \ge \sum_{T \in \cT_c} \int_T \int_0^1 
  \left( \frac{\partial a}{\partial \xi}(x, u_1, \nabla z(t)) \nabla w \right) \tdot \nabla v
 + \frac{\partial A}{\partial \eta }(x, z(t), \nabla u_2)w \nabla u_2  \tdot \nabla v
\nonumber \\
& + \frac{\partial b}{\partial \eta}(x,z(t)) w v \, \dif t \, \dif x,
\end{align} 
for $v$ given by Definition \ref{def:test}.  We next develop conditions on the 
discretization in one and two dimensions for which the above inequality cannot
hold.

\section{Results for one dimension}
\label{sec:1D}
Let $\Omega = (\alpha,\beta)$, with a subdivision
\begin{align}\label{eqn:subdivision1D}
\alpha = a_0 < a_1 < \ldots < a_{n-1} < a_n = \beta,
\end{align}
where the mesh spacing is not assumed to be uniform.
Define the intervals $\cI_k = (a_{k-1}, a_k)$, $k = 1, \ldots, n$, and 
let $h_k = a_k - a_{k-1}$, the length of each respective interval.
Then $\cT = \cup_{1 \le k \le n}\{\overline \cI_k\}$.
Let $v' = \dif v/\dif x$. 
In one dimension, 
for the mixed problem with Dirichlet conditions at $x=\beta$, with 
Neumann data $\psi(\alpha) \in \R$, 
the weak form \eqref{eqn:PDEweak} reduces to: 
find $u \in \cV \coloneqq \cV_{0,\beta}$ such that
\begin{align}\label{eqn:PDEweak1D_m}
 \int_\Omega a (x,u,u')v' + b(x,u)v \, \dif x= \psi(\alpha)v(\alpha) ~\tforall v \in \cV. 
\end{align}
For the pure Dirichlet problem, \eqref{eqn:PDEweak} reduces to:
find $u \in \cV \coloneqq \cV_0$ such that 
\begin{align}\label{eqn:PDEweak1D_d}
 \int_\Omega a (x,u,u')v' + b(x,u)v \, \dif x = 0, ~\tforall v \in \cV.
\end{align}
Without confusion, the dicrete space $\cV$ refers to either $\cV_{0,\beta}$,
containing the piecewise linear functions that vanish at $x = \beta$ for
problem \eqref{eqn:PDEweak1D_m}; or, $\cV_{0}$, 
containing functions that vanish at $x = \alpha$ and $x = \beta$ for problem 
\eqref{eqn:PDEweak1D_d}.

\begin{theorem}[One dimensional comparison theorem]
\label{thm:1Dcompare}
Let $u_1$ be a subsolution as in \eqref{eqn:subsoln} of either the mixed 
problem \eqref{eqn:PDEweak1D_m} or
the Dirichlet problem \eqref{eqn:PDEweak1D_d}; 
and, let $u_2$ be a supersolution as in \eqref{eqn:supsoln}, of the same problem.  
Assume the conditions of Assumption \ref{A1:GT}, and
\begin{align}\label{eqn:1Dcond}
\max_{1 \le k \le n} 
\left\{ |u_2(a_k) - u_2(a_{k-1})| + \left(\frac{B_\eta}{K_\eta}\right)h_k^2 \right\} 
< \frac{2 \gamma_a}{K_\eta}.
\end{align}
Then, it holds that $u_1 \le u_2$ in $\Omega$.
\end{theorem}
If the lower order term $b$ is independent of $u$, then $B_\eta = 0$, and 
the condition \eqref{eqn:1Dcond} is similar to that in \cite{PoZh17a}, for a more
general diffusion coefficient.  If, on the other hand, $B_\eta > 0$, global mesh 
condition is introduced, as $h_k < \sqrt{2 \gamma_a/ B_\eta}$ for all 
$k = 1, \ldots, n$, is a necessary
condition for satisfaction of \eqref{eqn:1Dcond}. 

The proof of Theorem \ref{thm:1Dcompare} follows by using the test function $v$ from 
Definition \ref{def:test}
to show the right hand side of \eqref{eqn:red_subsup} is strictly positive.
\begin{proof}
Assume $w = u_1 - u_2$, is positive somewhere in $\Omega$.  
Then $\cT_c$ is nonempty, and in one dimension, 
inequality \eqref{eqn:red_subsup} reduces to
\begin{align}\label{eqn:1dc_001}
0 \ge \sum_{\cI_{k} \in \cT_c} \int_{\cI_{k}}\int_0^1 \frac{\partial a}{\partial \xi}(x,u_1,z'(t))w'v' +
\frac{\partial A}{\partial \eta}(x,z(t), u')u'wv' + 
\frac{\partial b}{\partial \eta}(x,z(t))wv \, \dif t \, \dif x.
\end{align}
Proceed by bounding each term on the right hand side of \eqref{eqn:1dc_001}.
On each interval $\cI_{k} \in \cT_c$, $w$ changes sign, and by Definition~\ref{def:test}
the functions $w'$ and $v'$ are constants with the same sign. 
Then, the product $w'v' = |w(a_k) - w(a_{k-1})|/h_k^2$ on $\cI_{k}$, and it holds that
\begin{align}\label{eqn:1dc_002}
\int_{\cI_k} \int_0^1 \frac{\partial a}{\partial \xi}(x,u_1,z) w'v' \, \dif t\, \dif x 
& = \frac{|w(a_k) - w(a_{k-1})|}{h_k^2}\int_{\cI_k} 
\int_0^1 \frac{\partial a}{\partial \xi}(x,u_1,z) \, \dif t \, \dif x 
\nonumber \\
&\ge \frac{|w(a_k) - w(a_{k-1})|}{h_k} \gamma_a, 
\end{align}
where $\gamma_a$ is the constant from \eqref{eqn:A1:elliptic}.
For the second term of  \eqref{eqn:1dc_001}, it is useful to note that 
$\int_{\cI_k} |w| \le |w(a_k) - w(a_{k-1})|h_k/2$, as precisely one of $w(a_k)$ and 
$w(a_{k-1})$ must be positive.  Then
\begin{align}\label{eqn:1dc_003}
\int_{\cI_k} \int_0^1 \frac{\partial A}{\partial \eta}(x,z(t), u')u'wv' \, \dif t \, \dif x
& \ge \frac{-K_\eta |u_2(a_k) - u_2(a_{k-1}) |}{h_k^2}\int_{\cI_k}|w| \, \dif x
\nonumber \\ 
& \ge -\left(\frac{|w(a_k) - w(a_{k-1})|}{h_k}\right) 
             \frac{K_\eta |u_2(a_k) - u_2(a_{k-1})|}{2},
\end{align}  
where $K_\eta$ is the constant from \eqref{eqn:A1:a_eta}.
Each integral over last term of \eqref{eqn:1dc_001} satisfies
\begin{align}\label{eqn:1dc_004}
\int_{\cI_k} \int_0^1 \frac{\partial b}{\partial \eta}(x,z(t))wv \, \dif t \, \dif x
 \ge -B_\eta \int_{\cI_k} |w|\ \dif x 
 \ge -|w(a_k) - w(a_{k-1})| \frac{B_\eta h_k}{2}, 
\end{align}
where $B_\eta$ is the constant from \eqref{eqn:A1:b_eta}.
Putting \eqref{eqn:1dc_002}, \eqref{eqn:1dc_003} and \eqref{eqn:1dc_004} 
together into \eqref{eqn:1dc_001} yields
\begin{align*}
0 \ge \sum_{\cI_k \in \cT_c} \frac{|w(a_k) - w(a_{k-1})|}{h_k}
\left(\gamma_a - |u_2(a_k) - u_2(a_{k-1})|\frac{K_\eta}{2} 
               - h_k^2 \frac{B_\eta}{2}  \right) > 0,
\end{align*}
where the strict positivity in the last inequality holds under the condition 
\eqref{eqn:1Dcond}. This contradiction
establishes that $w = u_1 - u_2$ cannot be positive anywhere on $\Omega$. 
\end{proof}

As any solution $u$ to \eqref{eqn:PDEweak1D_m} or \eqref{eqn:PDEweak1D_d} is both
a subsolution and a supersolution, the uniqueness of solutions
follows, under the assumption
\begin{align}\label{eqn:1dcond_un}
\max_{1 \le k \le n} 
\left\{ |u(a_k) - u(a_{k-1})| + \left(\frac{B_\eta}{K_\eta}\right)h_k^2 \right\} 
< \frac{2 \gamma_a}{K_\eta}.
\end{align}
The constants $\gamma_a, B_\eta$ and $K_\eta$ are based purely on the problem data,
and if they are known or can be approximated for a given problem, then 
\eqref{eqn:1dcond_un} can be easily and efficiently checked, and used to determine
uniqueness of a given computed solution. It is important in particular for 
adaptive algorithms to have such a condition which assures the uniqueness of the
discrete solution without unavailable {\em a priori} knowledge. As demonstrated
by the counterexamples of \cite{AnCh96a} ({\em cf.} \cite{PoZh17a}), some conditions
on the discretization are indeed necessary to assure the uniqueness of the solution.
\section{Results for two dimensions}
\label{sec:2D}
We next establish the uniqueness of the piecewise linear finite element solution
to \eqref{eqn:PDEdiscrete}
in two dimensions, under Assumption \eqref{A1:GT}.
The simplicial mesh is assumed to be uniformly acute, and the smallest angle to be 
bounded away from zero.
\begin{assumption}[Mesh regularity]
\label{A2:meshregularity}
There are numbers $0 < t_{min} \le t_{max} < \pi/2$, 
for which the interior angles $\theta_i, ~i = 1,2,3$, 
of each $T \in \cT$ satisfy
\begin{align}\label{eqn:meshregularity}
t_{min} \le \theta_i \le t_{max}, ~ i = 1, 2, 3.
\end{align}
Define the quantities
\begin{align}\label{eqn:minmax} 
s_{min} = \sin(t_{min}), ~\an~ c_{min} = \cos(t_{max}).
\end{align}
\end{assumption}
The acuteness condition which states that angles are bounded below $\pi/2$, 
agrees with that in \cite{PoZh17a} 
for the simpler case of $a(x,\eta,\xi) = A(x,\eta)\xi$.  
In the following analysis, the condition that the angles are bounded away from zero 
is used to control the maximum ratio of edge-lengths in any triangle.

The relation between the measure of each element $T$, and the lengths of the 
sides are given by standard trigonometric descriptions.
For each $T \in \cT$, let $|T|$ denote the two-dimensional measure, or area.
For any two distict edges $e_i$ and $e_j$, the area
$|T| = |e_i||e_j| \sin \theta_k/2$, for $\theta_k$ the interior angle between edges
$e_i$ and $e_j$. This provides the useful formula
$|e_i||e_j|/|T| = 2/\sin \theta_k$.
Under Assumption \ref{A2:meshregularity}, the ratio of the sines of any pair of angles
in a triangle $T$ is bounded away from zero. Define the local constants
\begin{align}
\label{eqn:cTdef}
c_T & \coloneqq \min_{i,j = 1,2,3} \cos \theta_i,
\\ \label{eqn:sTdef}
s_T &\coloneqq \max_{i, j = 1,2,3} \sin \theta_{i},
\\ \label{eqn:regdef}
r_T &\coloneqq \min_{i, j = 1,2,3} \frac{\sin \theta_{i} }{\sin \theta_{j}},
\text{ for } \theta_i, ~ i = 1,2,3, ~\text{ the angles of } T.
\end{align}
The constant $r_T$ is used to relate the lengths of edges 
of triangle $T$ by
\begin{align}\label{eqn:edge_ineq}
r_T |e_i| \le |e_j| \le r_T^{-1}|e_i|, ~ i,j = 1,2,3.
\end{align}

Each vertex corresponding to a mesh degree of freedom, $a \in \cD$,
has coordinates $a = (x_1,x_2) \in \overline \Omega \setminus \Gamma_D$.  
From \S \ref{subsec:discretization}, 
$\cV \coloneqq \cV_{0,D}$ is the piecewise linear Lagrange finite element
space subordinate to partition $\cT$, that vanishes on $\Gamma_D$
in the sense of the trace. 

\subsection{Relations between gradients of basis functions}
\label{subsec:2Dsetup}
To clarify the techical lemmas that follow,
some standard notations and properties of piecewise linear finite elements
in two dimensions are now reviewed.
The following relations involving gradients of basis functions
are used often in the analysis.

Let $\{a_1, a_2, a_3 \}$ be a local counterclockwise numbering of the vertices 
of a simplex $T \in \cT$.  Let the corresponding edges $\{e_1, e_2, e_3 \}$, follow
a consistent local numbering, with edge $e_i$ opposite vertex 
$a_i, ~ i = 1, 2, 3$.
Let $\varphi_i$ be the basis function on element $T \in \cT$ defined by its nodal
values at the vertices of $T$.
\begin{align*}
\varphi_i(a_j) 
= \left\{ \begin{array}{ll}
  1, & i = j, \\
  0, & i \ne j.
  \end{array}\right., ~i,j = 1, 2, 3.
\end{align*}

The inner product between gradients of basis functions and their respective 
integrals over elements $T \in \cT$, may be
computed by change of variables to a reference element $\widehat T$, in reference 
domain variables $(\widehat x_1, \widehat x_2)$.
Specifically, the coordinates of $\widehat T$ are given as 
$\widehat a_1 = (0,0)^T$,
$\widehat a_2= (1, 0)^T$,
$\widehat a_3= (0,1)^T$.
The Jacobian of the transformation between reference coordinates
$\widehat x = (\widehat x_1, \widehat x_2)^T$, and physical coordinates 
$x = (x_1,x_2)^T$ , is given by
$J \widehat x = (x - a_1)$, with
$J = \left( \begin{array}{rrr}
    a_2-a_1 & a_3 - a_1
    \end{array}\right),$
for which $\det J = 2|T|$, with $|T|$ the area of triangle $T$.
The reference element $\widehat T$ is equipped 
with the nodal basis functions $\widehat \varphi_i, ~i = 1, 2, 3$, where
$ \widehat \varphi_1 = 1 - \widehat x_1 - \widehat x_2$,
$\widehat \varphi_2 = \widehat x_1$, 
$\widehat \varphi_3 = \widehat x_2$.
The gradients $\widehat \nabla$ are taken with respect to the reference domain variables
$\widehat x_1$ and $\widehat x_2$, and 
the transformation of gradients between the physical and reference
domains is given by $\nabla \varphi_i = J^{-T} \widehat \nabla \widehat \varphi_i.$ 
The gradients of basis functions satisfy the identity
\begin{align}\label{eqn:grad_id}
\nabla \varphi_i + \nabla \varphi_j = -\nabla \varphi_k,
\end{align} 
for any distinct assignment of $i,j$ and $k$ to the integers $\{1,2,3\}$.
This allows
the representation of $\nabla \varphi_i^T \nabla \varphi_i$ in terms of edge-length $|e_i|$.
The maximum interior angle $t_{\max}<\pi/2$ from Assumption \ref{A2:meshregularity}
then assures 
$\nabla \varphi_i^T \nabla \varphi_j <0$, for any $i \ne j$.
The inner products between gradients in each element $T$  satisfy the following identities:
\begin{align} \label{eqn:gradTgrad}
\nabla \varphi_i^T \nabla \varphi_i =   \frac 1 {4|T|^2} |e_i|^2, ~\an~
\nabla \varphi_i^T \nabla \varphi_j = \frac {-1} {4|T|^2} |e_i||e_j|\cos \theta_k, 
~i \ne j. 
\end{align}

\subsection{Additional assumptions for the 2D problem}\label{subsec:2Dassume}
We next establish estimates which demonstrate for any $T \in \cT_c$, given by
\eqref{eqn:Tcdef}, that
\begin{align}\label{eqn:2Dexpansion}
&\int_T \int_0^1 
  \left( \frac{\partial a}{\partial \xi}(x, u_1, \nabla z(t)) \nabla w\right) \tdot \nabla v
 + \left(\frac{\partial A}{\partial \eta }(x, z(t), \nabla u_2)\, w\right) \nabla u_2  \tdot \nabla v
\nonumber \\
& + \frac{\partial b}{\partial \eta}(x,z(t)) w v \, \dif t \, \dif x > 0,
\end{align}
with $v$ the test function given by Definition \ref{def:test}.
In light of \eqref{eqn:red_subsup}, 
this establishes by contradiction that $w = u_1 - u_2$ is nowhere positive.
To bound the leading term of \eqref{eqn:2Dexpansion} away from zero, 
some additional restrictions on 
the nonlinear diffusion coefficient $A$ are now considered.  
\begin{assumption}\label{A2:A}
Assume $A(x,\eta,\xi)$ is of the form
\begin{align}\label{eqn:A2:decomp}
A(x,\eta,\xi) = A_0(x,\eta) + A_1(x,\eta) f(|\xi|) + A_2(x) g(|\xi|).
\end{align}
Assume there is a positive constant $\lambda_0$, and there are nonnegative 
$\Lambda_1$ and $\Lambda_2$, with 
\begin{align}\label{eqn:A2:pos}
A_0(x,\eta) \ge \lambda_0,
\quad 0 \le A_1(x,\eta) \le \Lambda_1,
~\an~ 0 \le A_2(x) \le \Lambda_2,
\end{align}
for $\ale x \in \Omega$, and all $\eta \in \R$, and $\xi \in \R^2$. 

Assume $f(s), g(s) \ge 0$, and $f$ satisfies the following growth condition.
There is a constant $C_f$ with 
\begin{align}\label{eqn:fbound}
s|f'(s)| \le C_f, ~\tforall s \ge 0.
\end{align}
Assume $g$ satisfies one of the two following conditions.
\begin{align}\label{eqn:gbounded}
s|g'(s)| &\le C_g, ~\tforall s \ge 0, 
\\ \label{eqn:gunbdd}
s|g'(s)| &\le \widehat C_g g(s), ~\tforall s \ge 0, 
~\with~ 0 \le \widehat C_g \le c_{min}.
\end{align}
\end{assumption}
The function $g$ is not assumed to be either
bounded or bounded away from zero, while the boundedness of $f$ is required from 
\eqref{eqn:A1:a_eta} of Assumption \ref{A1:GT}. 
Functions $f$ and $g$ that satisfy Assumption \ref{A2:A} are not 
uncommon.  Some examples are given in the next remark.
\begin{remark}
Admissible functions $\phi(|\xi|)$ 
that satisfy \eqref{eqn:fbound}, \eqref{eqn:gbounded} include the following.
\begin{align*}
\phi(|\xi|) = (\kappa + |\xi|^2)^{-\alpha}, ~\text{ for }~ \kappa > 0 ~\an~ \alpha \ge 0,
\end{align*} 
which appears for instance as the diffusion coefficient in the
equation for capillarity (see \cite[Chapter 10]{GilbargTrudinger}, ) as well as 
the equations of prescribed mean curvature (see \cite{Verfurth94}), with
$\kappa = 1$ and $\alpha = 1/2$. 
\[
\phi(|\xi|) = 2\left(K_0 + \sqrt{K_0^2 + 4|\xi|} \right)^{-1}, \quad K_0 > 0,
\]
which is numerically investigated as a specific explicit case of the more general 
implicitly defined coefficient used in the modeling of glacier ice, as
analyzed in \cite{GlRa03}.
\[
\phi(|\xi|) = \arctan(|\xi|), ~\an~ \phi(|\xi|) = \tanh(|\xi|).
\]
Unbounded functions that satisfy \eqref{eqn:gbounded} include
\[
\phi(|\xi|) = \log( \kappa + |\xi|^2), ~ \kappa > 1,
\]
which allows for $g(|\xi|)$ hence $A(x,\eta,\xi)$ to be unbounded, 
albeit with slow growth.

Functions satisfying \eqref{eqn:gunbdd} include those of $p$-Laplacian type,
for $p$ close to $2$.
\[
\phi(|\xi|) = |\xi|^{p-2}, ~\text{ for } |p-2| < c_{min}.
\]
\end{remark}

\subsection{Technical lemmas in two dimensions}\label{subsec:2Dlemmas}
An important quantity in the analysis is the maximum
variance of a function over a given element.  For piecewise linear functions,
this is simply the maximum difference between any two vertex values on a given triangle.
\begin{definition}\label{def:delta}
For a function $\phi \in \cV$, and element $T \in \cT$, 
define $\delta_T(\phi)$ as follows. 
\begin{align}\label{eqn:def:delta}
\delta_T(\phi) = \max_{i,j = \{1,2, 3\}}|\phi(a_i) - \phi(a_j)|,
\end{align}
where $\{a_1, a_2, a_3\}$, are the vertices of $T$.
\end{definition}
In the technical lemmas which bound each term in the expansion
\eqref{eqn:2Dexpansion}, the following identity is used repeatedly.  
\begin{align}\label{eqn:gradid_32}
\nabla \phi &= \phi(a_i) \nabla \varphi_i + \phi(a_j) \nabla \varphi_j 
                                        + \phi(a_k) \nabla \varphi_k 
\nonumber \\
& = \phi(a_i) \nabla \varphi_i + \phi(a_j) \nabla(\varphi_j + \varphi_k ) 
                           + (\phi(a_k) - \phi(a_j)) \nabla \varphi_k
\nonumber \\
& = (\phi(a_i) - \phi(a_j)) \nabla \varphi_i + (\phi(a_k) - \phi(a_j)) \nabla \varphi_k.
\end{align}
The first Lemma characterizes the strict positivity of the first term of
\eqref{eqn:2Dexpansion}
\begin{lemma}\label{lemma:a_xi}
Let Assumptions \ref{A1:GT}, \ref{A2:meshregularity} and \ref{A2:A} hold, with $g$
satisfying \eqref{eqn:gbounded}.
Let $w, u \in \cV$, and $z(t) \in \cV, ~ 0 \le t \le 1$. 
Let $a_i$, $a_j$ and $a_k$ be the three vertices
of $T \in \cT_{c}$, ordered so that $w(a_i) \ge w(a_j) \ge w(a_k)$ with $w(a_{i}) >0$ 
and  $w(a_k) \le 0$. 
Let $v$ be given by Definition \ref{def:test}.
Assume there is a constant $p_T > 0$, for which the constants 
$\lambda_0, \Lambda_1$ and $\Lambda_2$ of \eqref{eqn:A2:pos}, and
$C_f$ and $C_g$ of \eqref{eqn:fbound} and \eqref{eqn:gbounded}, satisfy the relation
\begin{align}\label{eqn:pT}
p_T \coloneqq \lambda_0 \cos\theta_j - \Lambda_1 C_f - \Lambda_2 C_g > 0.
\end{align}
\begin{enumerate}
\item  If $w(a_j) \le 0$, namely $w$ is positive only at the vertex $a_{i}$, it holds that
\begin{align}\label{eqn:lem:e1_res1}
&\int_T \int_0^1 \nabla w^T \left( \frac{\partial a}{\partial \xi}
(x,u, \nabla z(t))\right)^T \nabla v \, \dif t \, \dif x 
\nonumber \\
&\ge \frac{1}{2 \sin \theta_j} 
\left\{ (w(a_i) - w(a_j))\frac{\gamma_a}{r_T} + (w(a_j) - w(a_k))p_T\right\}.
\end{align}
  \item If $w(a_j) \ge 0$, namely $w$ is positive at both 
$a_{i}$ and $a_{j}$,  it holds that 
\begin{align}\label{eqn:lem:e1:res2}
&\int_T \int_0^1 \nabla w^T \left( \frac{\partial a}{\partial \xi}
(x,u, \nabla z(t))\right)^T \nabla v \, \dif t \, \dif x 
\nonumber \\
&\ge \frac{1}{2 \sin \theta_j} 
\left\{ (w(a_i) - w(a_j))p_T + (w(a_j) - w(a_k))\frac{\gamma_a}{r_T}\right\}.
\end{align}
\end{enumerate}
\end{lemma}
\begin{proof}
First, expand $\nabla w$ as a linear combination of basis functions as in
\eqref{eqn:gradid_32}. For any $\nabla z \in \R^2$, abbreviating
$\partial a(x,u, \nabla z)/\partial \xi$ as $(\partial a / \partial \xi)$, and noting the 
structure of $a$ implies the symmetry of $\partial a/ \partial \xi$, we have
\begin{align}\label{eqn:e1_001}
\nabla w^T \left(\frac{\partial a}{\partial \xi}\right)^T \nabla v
& = (w(a_i) - w(a_j)) \nabla \varphi_i^T \left( \frac{\partial a}{\partial \xi} \right)  \nabla v
\nonumber \\
&  + (w(a_k) - w(a_j)) \nabla \varphi_k^T \left( \frac{\partial a}{\partial \xi} \right)  \nabla v.
\end{align}
In the case that $w$ has one positive vertex, $\nabla v = \nabla \varphi_i$, and 
in the case that $w$ has two positive vertices, $\nabla v = -\nabla \varphi_k$.
In the first case, the ellipticity condition \eqref{eqn:A1:elliptic} 
implies 
\begin{align}\label{eqn:e1_002}
(w(a_i) - w(a_j)) \nabla \varphi_i^T \left( \frac{\partial a}{\partial \xi} \right)  \nabla v
&\ge 
(w(a_i) - w(a_j)) \gamma_a \nabla \varphi_i^T \nabla \varphi_i,
\nonumber \\
&\ge (w(a_i) - w(a_j)) \frac{\gamma_a}{r_T} \frac{|e_i||e_k|}{4|T|^2}
\nonumber \\
& = (w(a_i) - w(a_j)) \frac{\gamma_a}{r_T} \frac{1}{2|T|\sin \theta_j},
\end{align}
where $r_T$ defined in \eqref{eqn:regdef} is used to relate the lengths of edges
$e_i$ and $e_k$.
In the second case, the same condition implies
\begin{align}\label{eqn:e1_003}
(w(a_k) - w(a_j)) \nabla \varphi_k^T \left( \frac{\partial a}{\partial \xi} \right)  \nabla v
&\ge (w(a_j) - w(a_k)) \gamma_a \nabla \varphi_k^T \nabla \varphi_k
\nonumber \\
& \ge (w(a_j) - w(a_k)) \frac{\gamma_a}{r_T} \frac{1}{2|T|\sin \theta_j}.
\end{align}
The above estimate for each case yields a strictly positive contribution. 
For the remaining term of \eqref{eqn:e1_001}, apply
the decomposition of Assumption \ref{A2:A}. 
\begin{align}\label{eqn:e1_004}
\nabla \varphi_i^T \left(\frac{\partial a}{\partial \xi} \right) \nabla \varphi_k
& = 
\nabla \varphi_i^T \nabla z \left( \frac{\partial A}{\partial \xi} \right) \nabla \varphi_k
+ A(x,u,\nabla z) \nabla \varphi_i^T \nabla \varphi_k
\nonumber \\
& = (\nabla \varphi_i^T \nabla z)
\left\{
  A_1(x,u) \frac{\partial f}{\partial \xi}(|\nabla z|) 
+ A_2(x) \frac{\partial g}{\partial \xi}(|\nabla z|) \right\}
  \nabla \varphi_k
\nonumber \\
&+ A(x,u,\nabla z) \nabla \varphi_i^T \nabla \varphi_k.
\end{align} 
The Jacobian of $f(|\xi|)$ (respectively, $g(|\xi|)$) has the form
\begin{align*}
\frac{\partial f}{\partial \xi}(|\xi|) = f'(|\xi|)|\xi|^{-1} \xi^T.
\end{align*}
The first term on the
right hand side of \eqref{eqn:e1_004} then satisfies
\begin{align}\label{eqn:e1_005}
 (\nabla \varphi_i^T \nabla z)A_1(x,u) \frac{\partial f}{\partial \xi}(|\nabla z|)\nabla \varphi_k
& = A_1(x,u)
    (\nabla \varphi_i^T \nabla z) f'(|\nabla z|) |\nabla z|^{-1} \nabla z^T \nabla \varphi_k
\nonumber \\
& \le \Lambda_1 (\nabla \varphi_i^T \nabla z) |f'(|\nabla z|)||\nabla \varphi_k|
\nonumber \\
& \le \Lambda_1 |\nabla \varphi_i|  |\nabla \varphi_k| |f'(|\nabla z|)||\nabla z|
\nonumber \\
& \le \Lambda_1 C_f|\nabla \varphi_i| |\nabla \varphi_k|,
\end{align}
where the last inequality follows from \eqref{eqn:fbound}.
Similarly for the second term on the right hand side of \eqref{eqn:e1_004}, it holds 
\begin{align}\label{eqn:e1_005g}
 (\nabla \varphi_i^T \nabla z)A_2(x) \frac{\partial g}{\partial \xi}(|\nabla z|)\nabla \varphi_k
& \le \Lambda_2 C_g|\nabla \varphi_i| |\nabla \varphi_k|.
\end{align}
With the Assumption~\ref{A2:A}, it is clear that $A(x, u, \nabla z) \ge \lambda_{0}$. Therefore the third term on the right hand side of \eqref{eqn:e1_004} satisfies
\begin{align}\label{eqn:e1_005a}
A(x,u,\nabla z) \nabla \varphi_i^T \nabla \varphi_k
& = -A(x, u, \nabla z)\frac{|e_i||e_k| \cos \theta_j}{4|T|^2} 
\le -\lambda_0 \frac{|e_i||e_k| \cos \theta_j}{4|T|^2}.
\end{align}
Applying  \eqref{eqn:e1_005}, \eqref{eqn:e1_005g} and \eqref{eqn:e1_005a} to 
\eqref{eqn:e1_004}, we obtain
\begin{align}\label{eqn:e1_007}
-\nabla \varphi_i^T \left(\frac{\partial a}{\partial \xi} \right) \nabla \varphi_k
& \ge \frac{|e_i||e_k|}{4|T|^2}(\lambda_0\cos\theta_j - \Lambda_1 C_f - \Lambda_2 C_g)
= \frac{p_T}{2 |T|\sin \theta_j},
\end{align}
where the sign on the left-hand side agrees with $(w(a_k) - w(a_j))$ in the case
of one positive vertex where $\nabla v = \nabla \varphi_i$; and 
$-(w(a_i) - w(a_j))$, in the case of two positive 
vertices, where $\nabla v = -\nabla \varphi_k.$
For the case of one positive vertex,
putting \eqref{eqn:e1_001} together with \eqref{eqn:e1_002} 
and \eqref{eqn:e1_007} and integrating, yields
\begin{align}\label{eqn:e1_008}
\int_T \int_0^1 \nabla w^T \left(\frac{\partial a}{\partial \xi}\right) \nabla v \, \dif t \, \dif x
& \ge \frac{1}{2 \sin \theta_j} 
\left\{(w(a_i) - w(a_j)) \frac{\gamma_a}{r_T} + 
       (w(a_j) - w(a_k))p_T \right\},
\end{align}
establishing \eqref{eqn:lem:e1_res1}.  The inequality \eqref{eqn:lem:e1:res2} follows 
similarly, replacing \eqref{eqn:e1_002} with \eqref{eqn:e1_003}.
\end{proof}
The next corollary shows the corresponding result if the condition on $g$, 
\eqref{eqn:gbounded}, is replaced by \eqref{eqn:gunbdd}, in Assumption \ref{A2:A}.
\begin{corollary}\label{cor:a_xi}
Let Assumptions \ref{A1:GT}, \ref{A2:meshregularity} and \ref{A2:A} hold, with $g$
satisfying \eqref{eqn:gunbdd}.
Let $w, u \in \cV$, and $z(t) \in \cV, ~ 0 \le t \le 1$. 
Let $a_i$, $a_j$ and $a_k$ be the three vertices
of $T \in \cT_{c}$, ordered so that $w(a_i) \ge w(a_j) \ge w(a_k)$ with $w(a_{i}) >0$ and  $w(a_k) \le 0$. 
Let $v$ be given by Definition \ref{def:test}.
Assume there is a constant $p_T > 0$, for which the constants 
$\lambda_0$ and $\Lambda_1$ of \eqref{eqn:A2:pos}, and
$C_f$ of \eqref{eqn:fbound}, satisfy the relation
\begin{align}\label{eqn:pT0}
p_T \coloneqq \lambda_0 \cos\theta_j - \Lambda_1 C_f > 0.
\end{align}
\begin{enumerate}
\item  If $w(a_j) \le 0$, namely $w$ is positive only at the vertex $a_{i}$, it holds that
\begin{align}\label{eqn:lem:e1_res10}
&\int_T \int_0^1 \nabla w^T \left( \frac{\partial a}{\partial \xi}
(x,u, \nabla z(t))\right)^T \nabla v \, \dif t \, \dif x 
\nonumber \\
&\ge \frac{1}{2 \sin \theta_j} 
\left\{ (w(a_i) - w(a_j))\frac{\lambda_0}{r_T} + (w(a_j) - w(a_k))p_T\right\}.
\end{align}
\item If $w(a_j) \ge 0$, namely $w$ is positive at both $a_{i}$ and $a_{j}$,  
it holds that 
\begin{align}\label{eqn:lem:e1:res20}
&\int_T \int_0^1 \nabla w^T \left( \frac{\partial a}{\partial \xi}
(x,u, \nabla z(t))\right)^T \nabla v \, \dif t \, \dif x 
\nonumber \\
&\ge \frac{1}{2 \sin \theta_j} 
\left\{ (w(a_i) - w(a_j))p_T + (w(a_j) - w(a_k))\frac{\lambda_0}{r_T}\right\}.
\end{align}
\end{enumerate}
\end{corollary}
The proof is similar to Lemma \ref{lemma:a_xi}, and the differences are summarized
below.
\begin{proof}
The estimates \eqref{eqn:e1_001}-\eqref{eqn:e1_005} remain unchanged, and
\eqref{eqn:e1_005g} is replaced by
\begin{align}\label{eqn:corg_003}
 (\nabla \varphi_i^T \nabla z)A_2(x) \frac{\partial g}{\partial \xi}(|\nabla z|)\nabla \varphi_k
& \le A_2(x) g'(|\nabla z|) \nabla \varphi_i^T \nabla z \nabla z^T \nabla \varphi_k 
  |\nabla z|^{-1}
\nonumber \\
& \le A_2(x) g'(|\nabla z|)|\nabla z| 
\nabla \varphi_i^T \left( \frac{\nabla z \nabla z^T}{\nabla z^T \nabla z} \right)
\nabla \varphi_k 
\nonumber \\
& \le A_2(x) \widehat C_g g(|\nabla z|) \frac{|e_i||e_k|}{4|T|^2}. 
\end{align}
The bound \eqref{eqn:e1_005a} is now replaced by
\begin{align}\label{eqn:corg_004}
A(x,u,\nabla z) \nabla \varphi_i^T \nabla \varphi_k 
& = -A(x,u,\nabla z) \frac{|e_i||e_k|\cos \theta_j}{4|T|^2}
\nonumber \\
& \le -( \lambda_0 + A_2(x)g(|\nabla z|)) \frac{|e_i||e_k|\cos \theta_j}{4|T|^2}.
\end{align}
Using \eqref{eqn:corg_003} and \eqref{eqn:corg_004} in place of 
\eqref{eqn:e1_005g} and \eqref{eqn:e1_005a},  in \eqref{eqn:e1_007} yields
\begin{align}\label{eqn:corg_005}
-\nabla \varphi_i^T \left(\frac{\partial a}{\partial \xi} \right) \nabla \varphi_k
& \ge \frac{|e_i||e_k|}{4|T|^2}\left(\lambda_0\cos\theta_j - \Lambda_1 C_f
+ A_2(x)g(|\xi|)(\cos \theta_j - \widehat C_g) \right)
\nonumber \\
& \ge \frac{|e_i||e_k|}{4|T|^2}(\lambda_0\cos\theta_j - \Lambda_1 C_f)
= \frac{p_T}{2 |T|\sin \theta_j},
\end{align}
under assumption \eqref{eqn:gunbdd}. The remainder of the proof remains unchanged.
\end{proof}
The second term of \eqref{eqn:2Dexpansion} is bounded by the estimates of
Lemma \ref{lemma:A_eta}.
These are similar to the ones found in \cite{PoZh17a}, 
where a Lipschitz assumption replaces the bound on the 
derivative $\partial A/ \partial \eta$. The key idea is to write $|w|$ as a multiple of 
$\delta_T(w) = w(a_i) - w(a_k)$, which can then be factored out of each term in the
expansion \eqref{eqn:2Dexpansion}. The positive part is given by the results of 
Lemma~\ref{lemma:a_xi}, and the parts that may not be positive are controlled 
by the  variance in the coefficients of $u_2$, which functions as a measurable control
as found in Lemma \ref{lemma:A_eta}; and, by the meshsize in the lower order term as given
in Lemma \ref{lemma:b_eta}.

\begin{lemma}\label{lemma:A_eta}
Let Assumptions \ref{A1:GT}, and \ref{A2:meshregularity} hold.
Let $w, u \in \cV$, and $z(t) \in \cV, ~ 0 \le t \le 1$. 
Let $a_i$, $a_j$ and $a_k$ be the three vertices
of $T \in \cT_{c}$, ordered so that $w(a_i) \ge w(a_j) \ge w(a_k)$ with $w(a_{i}) >0$ and  $w(a_k) \le 0$.
Let $v$ be given by Definition \ref{def:test}.
Then, it holds that
\begin{align}\label{eqn:lem:e2:res1}
\int_T \int_0^1
\frac{\partial A}{\partial \eta}(x,z(t),\nabla u) w \nabla u^T \nabla v \, \dif t \, \dif x 
\ge \frac{-\delta_T(w)\delta_T(u)}{2\sin \theta_j}\frac{7 K_\eta}{6} (1 + r_T^{-1}).
\end{align} 
\end{lemma}
\begin{proof}
In the case that $w$ has one positive vertex, $\nabla v = \nabla \varphi_i$.
Applying expansion \eqref{eqn:gradid_32} to $\nabla u$, followed by 
\eqref{eqn:regdef}, one finds
\begin{align}\label{eqn:e2_001}
\nabla u^T\nabla v & = (u(a_i) - u(a_j))\nabla \varphi_i^T \nabla \varphi_i
+ (u(a_k) - u(a_j)) \nabla \varphi_k^T \nabla \varphi_i 
\nonumber \\
& \le \delta_T(u)\frac{|e_i||e_k|}{4|T|^2}(1 + r_T^{-1})
\nonumber \\
& = \frac{\delta_T(u)}{2|T|\sin \theta_j} (1 + r_T^{-1}).
\end{align}
In the case that $w$ has two positive vertices, $\nabla v = -\nabla \varphi_k$, 
leading to the same resullt.
\begin{align}\label{eqn:e2_002}
\nabla u^T\nabla v & = -(u(a_i) - u(a_j))\nabla \varphi_i^T \nabla \varphi_k
- (u(a_k) - u(a_j)) \nabla \varphi_k^T \nabla \varphi_k 
\nonumber \\
& \le \frac{\delta_T(u)}{2|T|\sin \theta_j} (1 + r_T^{-1}).
\end{align}
Applying the bound \eqref{eqn:A1:a_eta} on $(\partial A/\partial \eta)$, then yields
\begin{align}\label{eqn:e2_003}
\int_T \int_0^1  \frac{\partial A}{\partial \eta} (x,z(t),\nabla u) w \nabla u^T \nabla v \, \dif t \, \dif x
& \le \frac{K_\eta \delta_T(u)}{2|T|\sin \theta_j} (1 + r_T^{-1})
  \int_T |w|\, \dif x.
\end{align}
As shown in \cite{PoZh17a}, and repeated here for convenience, 
the integral over $T$ of $|w|$, can be bounded in terms of $\delta_T(w)$
making use of $\varphi_i + \varphi_j + \varphi_k = 1$, and the ordering
$w(a_i) \ge w(a_j) \ge w(a_k)$.
\begin{align*}
|w| &= |w(a_i) \varphi_i + w(a_j) \varphi_j + w(a_k) \varphi_k |
  \nonumber \\
    &= |(w(a_i)-w(a_j)) \varphi_i + (w(a_k) -w(a_j)) \varphi_k + 
         w(a_j)(\varphi_i + \varphi_j + \varphi_k) |
  \nonumber \\
    &\le (w(a_i) - w(a_j))\varphi_i + (w(a_j) - w(a_k))\varphi_k + (w(a_j) - w(a_k))
  \nonumber \\
    &= (w(a_i) - w(a_j))\varphi_i + (w(a_j) - w(a_k))(1+\varphi_k). 
\end{align*}
Applying $\int_T \varphi_i \, \dif x = \int_T \varphi_k \, \dif x= |T|/6$, demonstrates
\begin{align}\label{eqn:e2_004}
\int_T|w| \, \dif x \le (w(a_i) - w(a_j)) \frac{|T|}{6} + (w(a_j) - w(a_k))\frac{7|T|}{6}
          \le \delta_T(w) \frac{7|T|}{6}.
\end{align}
Putting together \eqref{eqn:e2_003} and \eqref{eqn:e2_004}, yields the desired result.
\end{proof}

Finally, we consider a bound on the third term of \eqref{eqn:2Dexpansion}.
\begin{lemma}\label{lemma:b_eta}
Let Assumptions \ref{A1:GT}, and \ref{A2:meshregularity} hold.
Let $w, u \in \cV$, and $z(t) \in \cV, ~ 0 \le t \le 1$. 
Let $a_i$, $a_j$ and $a_k$ be the three vertices
of $T \in \cT_{c}$, ordered so that $w(a_i) \ge w(a_j) \ge w(a_k)$. Suppose $w(a_k) \le 0$.
Let $v$ be given by Definition \ref{def:test}.
Then, it holds that
\begin{align}\label{eqn:lem:e3}
\int_T \int_0^1 \frac{\partial b}{\partial \eta}(x,z(t))w v \, \dif x \, \dif t 
  \ge -\delta_T(w) \frac{ 7 B_\eta |T|}{6}.
\end{align}
\end{lemma}
\begin{proof}
Applying the condition \eqref{eqn:A1:b_eta} bounding $(\partial b/\partial \eta)$, 
and \eqref{eqn:e2_004} bounded $|w|$, reveals
\begin{align*}
\int_T \int_0^1 \frac{\partial b}{\partial \eta}(x,z(t))w v \, \dif x \, \dif t 
& \ge -B_\eta \int_T |w| \, \dif x 
  \ge -\delta_T(w)\frac{ 7 B_\eta |T|}{6}.
\end{align*}
\end{proof}
Notably, \eqref{eqn:lem:e3} can be controlled by the area $|T|$ in the numerator, 
rather than $\delta_T(u)$ as in the result of Lemma \ref{lemma:A_eta}.
Effectively, this introduces a global meshsize condition as in the 1D case if the lower
order term $b(x,u)$ appears in \eqref{eqn:PDEstrong}. 
\subsection{Comparison theorem in two dimensions}\label{subsec:2Dcompare}
We are now ready to combine the results of Lemmas \ref{lemma:a_xi}, \ref{lemma:A_eta}
and \ref{lemma:b_eta} to prove a discrete comparison theorem.
\begin{theorem}[Two dimensional comparison theorem]
\label{thm:2Dcompare}
Let $u_1 \in \cV$ be a subsolution of \eqref{eqn:PDEdiscrete} as in 
\eqref{eqn:subsoln}, and let $u_2 \in \cV$ be a supersolution of the same problem, 
 as in \eqref{eqn:supsoln}. Let $w = u_1 - u_2 \in \cV$.   
Let Assumptions \ref{A1:GT}, \ref{A2:meshregularity} and \ref{A2:A} hold, with 
$g$ satisfying \eqref{eqn:gbounded}.
Assume $\lambda_0$, $\Lambda_1$, $\Lambda_2$ and $C_{f}, C_g$ of 
Assumption \ref{A2:A}, 
and $c_{min}$ of \eqref{eqn:minmax} satisfy the relation
\begin{align}\label{eqn:pOmega}
\lambda_0 c_{min} - \Lambda_1 C_f -\Lambda_2 C_g> 0.
\end{align}
Define the positive constant for each $T \in \cT$
\begin{align}\label{eqn:pstar}
p^\ast_T \coloneqq \min\{\lambda_0 c_T - \Lambda_1 C_f -\Lambda_2 C_g, \gamma_a/r_T\},
\end{align}
with $\gamma_a$ ~from \eqref{A1:GT}, $c_T$ from \eqref{eqn:cTdef}
and $r_T$ from \eqref{eqn:edge_ineq}.
Then, the satisfaction of the condition
\begin{align}\label{eqn:2Dcond}
\min_{T \in \cT} 
\left\{ 
p^\ast_T - \delta_T(u_2) \frac{7 K_\eta (1 + r_T^{-1})}{6} - \frac{7 B_\eta |T|s_T}{3}
\right\} > 0, 
\end{align}
with $s_T$ from \eqref{eqn:sTdef}, implies that $u_1 \le u_2$ in $\Omega$.
\end{theorem}
\begin{proof}
Assume $w = u_1 - u_2$ is positive somewhere in $\Omega$. This implies $w(a) > 0$ 
for some vertex $a \in \cD$.  
Let the test function $v\in\cV^+$ be given by Definition \ref{def:test}.
Then, from \eqref{eqn:red_subsup}, it holds that
\begin{align}\label{eqn:t2_001}
&  \sum_{T \in \cT_c}\int_T \int_0^1 
  \left( \frac{\partial a}{\partial \xi}(x, u_1, \nabla z(t)) \nabla w\right) \tdot \nabla v
 + \left( \frac{\partial A}{\partial \eta }(x, z(t), \nabla u_2)\, w \right) \nabla u_2  \tdot \nabla v
\nonumber \\
&+ \frac{\partial b}{\partial \eta}(x,z(t)) w v \, \dif t \, \dif x \le 0,
\end{align}
where $\cT_c$ defined in \eqref{eqn:Tcdef} 
is the set of all elements $T$ where $w$ is positive on either one or two vertices.
The hypothesis \eqref{eqn:pOmega} together with Lemma \ref{lemma:a_xi} implies 
for any $T \in \cT_c$, it holds that
\begin{align}\label{eqn:t2_002}
\int_T \int_0^1
\left( \frac{\partial a}{\partial \xi}(x, u_1, \nabla z(t)) \nabla w\right) \tdot \nabla v
\, \dif t \, \dif x
\ge \frac {\delta_T(w)}{2 \sin \theta_{T,j}} p^\ast_T,
\end{align}
where $\theta_{T,j}$ refers to $\theta_j$ of triangle $T$ with respect to the 
local indexing,
where $a_i, a_j$ and $a_k$ are the three vertices of $T$, ordered so that
$w(a_i) \ge w(a_j) \ge w(a_k)$.

Lemma \ref{lemma:b_eta}, together with the inequality $\sin\theta_j \le s_T$, where
$s_T$ is the sine of the maximum angle of $T$ as in \eqref{eqn:sTdef}, shows for
any $T \in \cT_c$ that
\begin{align}\label{eqn:t2_003}
\int_T \int_0^1 \frac{\partial b}{\partial \eta}(x,z(t)) w v \, \dif t \, \dif x 
\ge -\left(\frac{\delta_T(w)}{2 \sin \theta_{T,j}}\right) \frac{7 B_\eta |T|s_T}{3}.
\end{align}
Putting \eqref{eqn:t2_002} and \eqref{eqn:t2_003} and the result of 
Lemma \ref{lemma:A_eta} together into \eqref{eqn:red_subsup} yields
\begin{align}\label{eqn:t2_004}
&\int_{\Omega} (a(x,u_1,\nabla u_1) - a(x,u_2,\nabla u_2)) \cdot \nabla v 
  + (b(x, u_1)v - b(x,u_2) )v \, \dif x 
  \nonumber \\ 
  & \ge \sum_{T \in \cT_c} 
\frac{\delta_T(w)}{2 \sin \theta_{T,j}} \left\{
p^\ast_T - \delta_T(u_2) \frac{7 K_\eta (1 + r_T^{-1})}{6} - \frac{7 B_\eta |T|s_T}{3}
\right\} > 0.
\end{align} 
The positivity of \eqref{eqn:t2_004} is in direct contradiction to the nonpositivity
from \eqref{eqn:t2_001}, repeated from \eqref{eqn:red_subsup}.  This demonstrates 
that under the hypotheses of the theorem, the function $v$ must be nowhere positive,
which requires $u_1 \le u_2$ in $\Omega$.
\end{proof}
Replacing Lemmas \ref{lemma:a_xi} with Corallary \ref{cor:a_xi} and \ref{lemma:b_eta} 
allows us to prove a second comparison result.
\begin{corollary}
\label{cor:2Dcompare}
Let $u_1 \in \cV$ be a subsolution of \eqref{eqn:PDEdiscrete} as in 
\eqref{eqn:subsoln}, and let $u_2 \in \cV$ be a supersolution of the same problem, 
 as in \eqref{eqn:supsoln}. Let $w = u_1 - u_2 \in \cV$.   
Let Assumptions \ref{A1:GT}, \ref{A2:meshregularity} and \ref{A2:A} hold, with 
$g$ satisfying \eqref{eqn:gunbdd}.
Assume $\lambda_0, \Lambda_1$ and $C_{f}$ of Assumption \ref{A2:A}, 
and $c_{min}$ of \eqref{eqn:minmax} satisfy the relation
\begin{align*}
\lambda_0 c_{min} - \Lambda_1 > 0.
\end{align*}
Define the positive constant for each $T \in \cT$
\begin{align}\label{eqn:pstar_c}
p^\ast_T \coloneqq \min\{\lambda_0 c_T - \Lambda_1 C_f, \gamma_a/r_T\},
\end{align}
with $\gamma_a$ ~from \eqref{A1:GT}, $c_T$ from \eqref{eqn:cTdef}
and $r_T$ from \eqref{eqn:edge_ineq}.
Then, the satisfaction of the condition
\begin{align*}
\min_{T \in \cT} 
\left\{ 
p^\ast_T - \delta_T(u_2) \frac{7 K_\eta (1 + r_T^{-1})}{6} - \frac{7 B_\eta |T|s_T}{3}
\right\} > 0, 
\end{align*}
with $s_T$ from \eqref{eqn:sTdef}, implies that $u_1 \le u_2$ in $\Omega$.
\end{corollary}
\begin{proof}
The proof follows directly by replacing Lemma \ref{lemma:a_xi} by 
Corollary \ref{cor:a_xi} in \eqref{eqn:t2_002} 
of the proof of Theorem \ref{thm:2Dcompare}.
\end{proof}

\begin{remark}[Uniqueness of finite element solutions]
An important consequence of the comparison theorem is the uniqueness of solutions
to \eqref{eqn:PDEdiscrete}, which as demonstrated holds in two dimensions under the 
hypotheses of Theorem \ref{thm:2Dcompare}, under the condition
\[
\min_{T \in \cT}
\left\{
p^\ast_T - \delta_T(u) \frac{7 K_\eta (1 + r_T^{-1})}{6} - \frac{7 B_\eta |T|s_T}{3} 
\right\} > 0,
\]
with $p_T^\ast$ given either by \eqref{eqn:pstar} or \eqref{eqn:pstar_c}.
The quantities involved to verify this condition consist of global constants bounding 
the problem data and local quantities characterizing the triangulation and the 
computed solution $u$. The global constants are
$\lambda_0, \Lambda_1, \Lambda_2$
and $C_f, C_g$ or $\widehat C_g$ of Assumption \ref{A2:A} 
and $\gamma_a, K_\eta, B_\eta$ of Assumption \ref{A1:GT}. 
The necessary triangulation data describes
the area $|T|$ and the smallest and largest angles of each element $T \in \cT$: 
$c_T$, $r_T$ and $s_T$ of \eqref{eqn:cTdef}-\eqref{eqn:regdef}. 
Finally, it is required to check the greatest
difference between nodal values of the computed solution on each element, $\delta_T(u)$.
All these quantities  can be easily and efficiently computed in practice. 
\end{remark}
Notably, the global meshsize condition comes only from the lower order term,
and solutions to the pure diffusion problem can be demonstrated unique without
a globally small meshsize.  Essentially, the meshsize needs to be small where
the gradient is large.

\section{A Semilinear problem}\label{sec:semilinear}
In this section, we consider the discrete comparison principle for a 
special case of the problem class \eqref{eqn:PDEstrong}, the semilinear problem:
\begin{align}\label{eqn:PDEsemiL}
-\Delta u  + b(x,u) = 0 ~\inn~ \Omega \subset \R^d, ~ u = 0 \on \partial \Omega.
\end{align}
For simplicity, we consider the homogeneous Dirichlet problem in one and two dimensions.
The nonlinearity $b(x,u)$ is assumed to satisfy the requirements of Assumption
\ref{A1:GT}. The  discrete version of problem \eqref{eqn:PDEsemiL} is:
Find $u \in \cV \subset H_0^1(\Omega)$ such that
\begin{align}\label{eqn:semiDiscrete}
\int_\Omega \nabla u^T \nabla v + b(x,u) v \dif x= 0, \tforall v \in \cV.
\end{align}
Based on the previous sections, 
we can obtain the following discrete comparison result for 
\eqref{eqn:semiDiscrete}, which is a simplified version of 
Theorem~\ref {thm:1Dcompare} and Theorem~\ref{thm:2Dcompare} in the semilinear case.
However, we find this technique leads to a suboptimal mesh condition.  We then
improve the condition with a linear algebra argument in Theorem \ref{thm:semi}.  
While the techniques of Theorem \ref{thm:semi} do not apply to the quasilinear problem
\eqref{eqn:PDEstrong}, they suggest sharper criteria for comparison 
theorems and uniqueness may be attainable.  We include both approaches for the 
semilinear problem \eqref{eqn:PDEsemiL} for completeness.
\begin{theorem}
\label{thm:semilinear}
Let $u_1 \in \cV$ be a subsolution of \eqref{eqn:semiDiscrete},
and let $u_2 \in \cV$ be a supersolution of \eqref{eqn:semiDiscrete}.  
Let $w = u_1 - u_2 \in \cV$.  Let $b(x,u)$ satisfy the Assumption \ref{A1:GT}, 
and for the 2D problem, let the partition satisfy Assumption \ref{A2:meshregularity}.
Under the respective conditions for the 1D and 2D problems:
\begin{align}
\label{eqn:1DsemiCond0}
h_k^2 &< \frac{2}{B_\eta}, ~ k = 1, 2, \ldots, n, ~\text{ for }~ d = 1,
\\\label{eqn:2DsemiCond0}
|T| &< \frac{3}{7B_\eta} \min_{k = 1, 2,3}\cot\theta_{T,k}, 
~\text{ for each } T \in \cT, ~\text{ for }~ d=2,
\end{align}
it holds that $u_{1}\le u_{2}$ in $\Omega.$
\end{theorem}
\begin{proof}
We proceed by contradiction. 
Assume $w = u_{1} - u_{2}$ is positive on at least one vertex of $\cT$. 
Then $w$ changes signs on each $T\in \cT_c$, which must be nonempty. 
Let $v$ be defined as in Definition~\ref{def:test}. 
	
In the 1D case, similar to Theorem~\ref {thm:1Dcompare} on each $\cI_{k} \in \cT_{c}$, 
the product $w'v' = |w(a_{k}) - w(a_{k-1})|/h_{k}^{2}$. 
Thus by condition \eqref{eqn:1DsemiCond0} we have 
\[
\int_{\cI_{k}} w' v' + (b(x, u_{1}) -b(x, u_{2}))v \, \dif x 
\ge |w(a_{k}) - w(a_{k-1})|
\left(\frac 1{h_{k}} - \frac{B_{\eta}h_{k}}{2} \right) >0.
\]
This contradicts the condition that $u_{1}$ and $u_{2}$ are sub- and supersolutions of 
\eqref{eqn:semiDiscrete}.
	
In the 2D case, on each $T\in \cT_{c}$, label the vertices 
$a_{i}$, $a_{j}$ and $a_{k}$ such that $w(a_{i}) \ge w(a_{j}) \ge w(a_{k})$. 
Then with the Assumption \ref{A2:meshregularity}, it holds for the case
$w(a_j) \le 0$, that
\begin{align*}
\int_{T} \nabla w \nabla  v \, \dif x 
&= (w(a_{i}) -w(a_{j})) \int_{T}\nabla \varphi_{i}^{T}\nabla \varphi_{i} \, \dif x 
  + (w(a_{k}) -w(a_{j})) \int_{T}\nabla \varphi_{k}^{T}\nabla \varphi_{i} \,  \dif x\\
&=\frac{1}{2}(w(a_{i}) -w(a_{j})) (\cot \theta_{k} + \cot \theta_{j}) 
 + \frac{1}{2}(w(a_{j})-w(a_{k}))\cot \theta_{j}\\
&\ge \frac{1}{2}(w(a_{i}) - w(a_{k})) \cot \theta_{j}\\
&=\frac{1}{2}\delta_{T}(w) \cot \theta_{j}.
\end{align*}
As in Lemma~\ref{lemma:a_xi}, the case $w(a_j) > 0$ follows similarly.
By Lemma~\ref{lemma:b_eta} and \eqref{eqn:2DsemiCond0}, we have 
\[
\int_{T} \nabla w \nabla v  \, \dif x 
 + \int_{T}(b(x, u_{1}) -b(x,u_{2})) v \, \dif x 
\ge \frac{1}{2}\delta_{T}(w)\left(\cot \theta_{j} - \frac{7B_{\eta} |T|}{3}\right) 
>0.
\]
Under \eqref{eqn:2DsemiCond0} this yields a contradiction, establishing the result.
\end{proof}

A more direct linear algebraic approach to determine a discrete comparison 
principle which implies the uniqueness of \eqref{eqn:semiDiscrete} is next demonstrated. 
We can derive the discrete comparison
principle by considering a discrete maximum principle for the difference
$w = u_1 - u_2$, where $u_1\in \cV$ is a subsolution of \eqref{eqn:semiDiscrete}, and 
$u_2\in \cV$ is a supersolution of \eqref{eqn:semiDiscrete}.
Similarly to \eqref{eqn:decomp_T}, the piecewise linear $w \in \cV$ satisfies 
\begin{align}\label{eqn:semi_001}
&\int_\Omega \nabla w ^T \nabla v \, \dif x + 
 \int_\Omega \int_{0}^{1}\frac{\partial b}{\partial \eta}(x,z(t)) w v \,\dif t \, \dif x 
= \int_T f_\delta  v \, \dif x, \tforall v \in \cV,
\end{align}
where $z(t) = t u_1 + (1-t) u_2$, and $f_\delta$ is some nonpositive
$L^2$ integrable function defined by the left hand side of \eqref{eqn:semi_001}.
Clearly $f_\delta$ satisfies
$
\int_\Omega f_\delta v \, \dif x \le 0,~\tforall v \in \cV^+.
$
Equation \eqref{eqn:semi_001} is a linear reaction-diffusion equation with a 
bounded, nonnegative reaction term $c(x) = \int_0^1 \partial b/\partial \eta (x,z(t))\, \dif t$.
It is immatieral that the reaction term $c(x)$ is not explicitly available.
As such, the maximum principle in \S 3. of \cite{BrKoKr08} applies, establishing 
under the appropriate mesh conditions that $w \le 0$ on $\Omega$, hence $u_1 \le u_2$. 
To make this article self-contained, the argument of \cite{BrKoKr08} is summarized below.

Let $n_{dof}$ be the cardinality of $\cD$, the number of interior vertices of $\cT$.
The approximation $w \in \cV$ is a linear combination of basis functions given by
$
w = \sum_{i = 1}^{n_{dof}} W_i \varphi_i,
$
with $W = (W_1, \ldots W_{n_{dof}})^T$ the corresponding vector of coefficients.
The discrete form of the problem \eqref{eqn:semiDiscrete} is recovered by the solution
to the linear algebra system
\begin{align}\label{eqn:semi_003}
A W = F, ~\with~ A = S + M, \quad A = (a_{ij}), \quad M = (m_{ij}), \quad F = (f_j),
\end{align}
for stiffness matrix $S$ and mass matrix $M$ defined entrywise by
\begin{align}\label{eqn:semi_004}
s_{ij}&=s_{ji}=\sum_{T \in \cT} \int_T \nabla \varphi_i^T \nabla \varphi_j \, \dif x,
\nonumber \\ 
m_{ij}&=m_{ji}=\sum_{T \in \cT} \int_T \int_0^1 \frac{\partial b}{\partial \eta}(x,z(t)) 
        \varphi_i \varphi_j \, \dif t \, \dif x.
\end{align}
The load vector is given by
$f_j   = \sum_{T \in \cT} \int_T f_\delta \varphi_j \, \dif x$. Each 
$f_j$ is nonpositive, from \eqref{eqn:semi_001}. 

From \eqref{eqn:semi_003}, it is sufficient to show that $A^{-1}$ is entrywise 
nonnegative, to establish that each $W_{j}$ is nonpositive, from which it follows that 
$w \le 0$ and $u_1 \le u_2$.  
This is established by showing $A$ is a \emph{Stieltjes matrix}, meaning
$A$ is symmetric positive definite with nonpositive off-diagonal entries
(see for example \cite[Definition 3.23]{Varga.R2000}). 
\begin{remark}
As mentioned in \cite[\S 3.]{BrKoKr08}, it is easier 
to show $A$ is a Stieltjes matrix than an $M$ matrix, as it is not necessary to 
show irreducibility.  It also makes the current argument unsuitable for the full 
quasilinear problem \eqref{eqn:PDEstrong}, as the resulting linearized equation
would induce a nonsymmetric matrix.
\end{remark} 

The next theorem is a restatement of \cite[Theorem 3.7]{BrKoKr08}, reframed 
in the present context.
\begin{theorem}\label{thm:semi}
Let $u_1 \in \cV$ be a subsolution of \eqref{eqn:semiDiscrete},
and let $u_2 \in \cV$ be a supersolution of \eqref{eqn:semiDiscrete}.  
Let $w = u_1 - u_2 \in \cV$.  Let $b(x,u)$ satisfy the Assumption \ref{A1:GT}, 
and for the 2D problem, let the partition satisfy Assumption \ref{A2:meshregularity}.
Then $A$ as given in \eqref{eqn:semi_003}-\eqref{eqn:semi_004} 
is a Stieltjes matrix, under the 
respective conditions for the 1D and 2D problems.
\begin{align}\label{eqn:2DsemiCond}
|T| &\le \frac{6}{B_\eta} \min_{k = 1, 2,3}\cot\theta_{T,k}, 
~\text{ for each } T \in \cT, ~\text{ for }~ d=2,
\\ \label{eqn:1DsemiCond}
h_k^2 &\le \frac{6}{B_\eta}, ~ k = 1, 2, \ldots, n, ~\text{ for }~ d = 1.
\end{align}
\end{theorem}
\begin{proof}
In $d$ dimensions $\int_T \varphi_i \varphi_j \, \dif x = |T|/(d+1)(d+2)$ for $i \ne j$,
so the summands of the off-diagonal entries of $M$ satisfy
\begin{align}\label{eqn:semi_005}
\int_T \int_0^1\frac{\partial b}{\partial \eta}(x,z(t))\varphi_i \varphi_j \, \dif t \, \dif x
\le  \frac{B_\eta|T|}{(d+1)(d+2)}, &~i \ne j,
\end{align}
and the diagonal entries of $M$ are nonnegative. 
By \eqref{eqn:gradTgrad} in 2D (trivially in 1D). 
The diagonal entries of $S$ are positive.
The off-diagonal entries of $S$ constructed by \eqref{eqn:semi_004}, satisfy
$\int_T \nabla \varphi_i^T \nabla \varphi_j \, \dif x  = -\cot \theta_{T,i,j}/2$ in 2D, 
where $\theta_{T,i,j}$ is the angle of triangle $T$ between between edges $e_i$ and 
$e_j$. In the 1D case, $\int_{\cI_k} \varphi_k'\varphi_{k-1}' \, \dif x = -1/h_k$.

Under the conditions \eqref{eqn:2DsemiCond} (respectively, \eqref{eqn:1DsemiCond}),
and the construction \eqref{eqn:semi_003}-\eqref{eqn:semi_004}, the
matrix $A$ is symmetric positive definite with positive diagonal and nonpositive 
off-diagonal entries.  Hence it is a Stieltjes matrix.
\end{proof}
It follows directly from Theorem \ref{thm:semi} and \eqref{eqn:semi_001} 
that the solution $W$ to $AW=F$ is nonpositive,
so that $w = u_1 - u_2 \le 0$. This method of proof is preferred for the semilinear
problem \eqref{eqn:PDEsemiL}, as it gives an improved constant in the mesh condition.
While it does not apply directly to the quasilinear problem \eqref{eqn:PDEstrong},
a variant using an $M$-matrix or otherwise nonsymmetric monotone matrix may be applicable.

\section{Conclusion}
In this paper, we proved comparision theorems in 1D and 2D for elliptic quasilinear
diffusion problems discretized by standard $P_1$ finite elements, significantly
extending the results of \cite{PoZh17a}.  
We found the discrete comparision principles hold based on conditions relating
the given problem data, information about the area and angles of the mesh, and the 
variance of the computed solution over each mesh element.
The proofs are more complicated than the comparison theorem for the
continuous problem, the main setback being the positive part in the difference
of two solutions does not lie in the finite element space.  
There remains a significant gap between the class of problems for which comparision
principles hold for the PDE and for the corresponding discrete problem.  
For the class of problems investigated here, the discrete comparison principle implies
the uniqueness of the solution to the discrete problem, based on efficiently computable
conditions.  These results are useful for $h$-adaptive algorithms, where the 
mesh presumably remains coarse away from steep gradients in the solution or 
(near) singularities in the data.
\label{sec:conclusion}

\bibliographystyle{abbrv}
\bibliography{refsTRN,M}

\end{document}